\documentclass[11pt]{amsart}
\usepackage{amsmath,amssymb,amsthm,latexsym,graphicx}
\newtheorem{theorem}{\sc Theorem}[section]
\newtheorem{lemma}[theorem]{\sc Lemma}
\newtheorem{proposition}[theorem]{\sc Proposition}
\newtheorem{corollary}[theorem]{\sc Corollary}
\newtheorem{assumption}[theorem]{\sc Assumption}
\newtheorem{remark}[theorem]{\sc Remark}

\newcommand{\be}{\begin{equation}}
\newcommand{\ee}{\end{equation}}

\allowdisplaybreaks[1]

\begin{document}
\title[Fluctuations of quenched mean]{ Fluctuations of the quenched mean of a planar random walk in an i.i.d. random environment with forbidden direction}
\author[M.~Joseph]{Mathew Joseph}
\address{Mathew Joseph\\ University of Wisconsin-Madison\\ 
Mathematics Department\\ Van Vleck Hall\\ 480 Lincoln Dr.\\  
Madison WI 53706-1388\\ USA.}
\email{joseph@math.wisc.edu}
\urladdr{http://www.math.wisc.edu/~joseph}
\date{\today}
\begin{abstract}
We consider an i.i.d. random environment with a strong form of transience on the two dimensional integer lattice. Namely, the walk always moves forward in the $y$-direction. We prove a functional CLT for the quenched expected position of the random walk indexed by its level crossing times. We begin with a variation of the Martingale Central Limit Theorem. The main part of the paper checks the conditions of the theorem for our problem.
\end{abstract}
\maketitle









\section{Introduction}
One of several models in the study of random media is random walks in random environment (RWRE). An overview of this topic can be found in the lecture notes by Sznitman \cite{szn} and Zeitouni \cite{zet}. While one dimensional RWRE is fairly well understood, there are still many simple questions (transience/recurrence, law of large numbers, central limit theorems) about multidimensional RWRE which have not been resolved. In recent years, much progress has been made in the study of multidimensional RWRE but it is still far from complete. Let us now describe the model.

Let $\Omega = \{\omega=(\omega_{x \cdot})_{x \in \mathbb{Z}^d}:\omega_{x \cdot}=(\omega_{x,y})_{y \in \mathbb{Z}^d} \in [0,1]^{\mathbb{Z}^d}, \sum_y \omega_{x,y}=1 \}$ be the set of transition probabilities for different sites $x \in \mathbb{Z}^d$. Let $T_z$ denote the natural shifts on $\Omega$ so that $(T_z\omega)_{x\cdot}=\omega_{x+z,\cdot}$. $T_z$ can be viewed as shifting the origin to $z$. Let $\mathcal{S}$ be the product $\sigma$-field on the set $\Omega$. A probability measure $\mathbb{P}$ over $\Omega$ is chosen so that $(\Omega,\mathcal{S},\mathbb{P},(T_z)_{z \in \mathbb{Z}^d})$ is stationary and ergodic. The set $\Omega$ is the environment space and $\mathbb{P}$ gives a probability measure on the set of environments. Hence the name ``random environment". For each $\omega \in \Omega$ and $x \in \mathbb{Z}^d$, define a Markov chain $(X_n)_{n \ge 0}$ on $\mathbb{Z}^d$ and a probability measure $P_x^{\omega}$ on the sequence space such that
  \[P_x^{\omega}(X_0 = x)=1, \qquad P_x^{\omega}(X_{n+1}=z\vert X_n=y)= \omega_{y,z-y} \,\, \mbox{ for } y,z \in \mathbb{Z}^d. \]
 There are thus two steps involved. First the environment $\omega$ is chosen at random according to the probability measure $\mathbb{P}$ and then we have the ``random walk" with transition probabilities $P_x^{\omega}$ (assume $x$ is fixed beforehand). $P_x^{\omega}(\cdot)$  gives a probability measure on the space $(\mathbb{Z}^d)^{\mathbb{N}}$ and is called the \textit{quenched measure}. The \textit{averaged measure} is 
 \[ P_{x}( (X_n)_{n \ge 0} \in A) =\int_{\Omega} P_x^{\omega}((X_n)_{n \ge 0} \in A)\mathbb{P}(d\omega).\]

Other than the behaviour of the walks itself, a natural quantity of interest is the \textit{quenched mean} $E_x^{\omega}(X_n)= \sum_z z P_x^{\omega}(X_n=z)$, i.e. the average position of the walk in $n$ steps given the environment $\omega$. Notice that as a function of $\omega$, this is a random variable. A question of interest would be a CLT for the quenched mean. A handful of papers have dealt with this subject. Bernabei~\cite{ber} and Boldrighini, Pellegrinotti~\cite{bp} deal with this question in the case where $\mathbb{P}$ is assumed to be i.i.d. and there is a direction in which the walk moves deterministically one step upwards (time direction). Bernabei~\cite{ber} showed that the centered quenched mean, normalised by its standard deviation, converges to a normal random variable and he also showed that the standard deviation is of order $n^{\frac{1}{4}}$. In ~\cite{bp}, the authors prove a central limit theorem for the correction caused by the random environment on the mean of a test function. Both these papers however assume that there are only finitely many transition probabilities. Bal\'azs, Rassoul-Agha and Sepp\"al\"ainen~\cite{brs} replace the assumption of finitely many transition probabilites with a ''finite range" assumption for $d=2$ to prove an invariance principle for the quenched mean. In this paper we restrict to $d=2$ and look at the case where the walk is allowed to make larger steps upward. We prove a functional CLT for the position of the walk on crossing level $n$; there is a nice martingale structure in the background as will be evident in the proof.

Another reason for looking at the quenched mean is that recently a number of papers by Rassoul-Agha and Sepp\"al\"ainen (see \cite{rsptrf}, \cite{rsqi}) prove \textit{quenched} CLT's for $X_n$ using subdiffusivity of the quenched mean. Let us now describe our model.

\bigskip

 \textbf{Model:} We restrict ourselves to dimension $2$. The environment is assumed to be i.i.d. over the different sites. The walk is forced to move at least one step in the $e_2=(0,1)$ direction; this is a special case of walks with forbidden direction (\cite{rsqi}). We assume a finite range on the steps of the walk and also an ellipticity condition .\\

\begin{center}\includegraphics[height=1in]{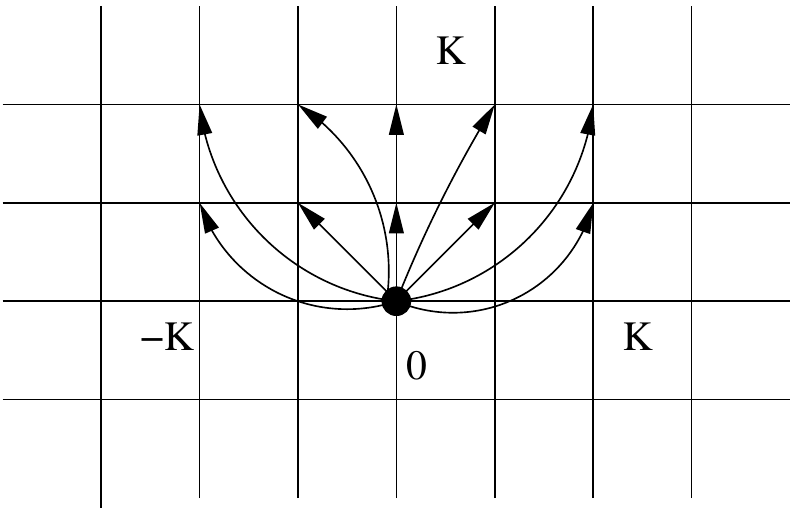} \end{center}

\begin{assumption}
\label{model} 
 (i)$\mathbb{P}$ is i.i.d. over $x \in \mathbb{Z}^2$ ($\mathbb{P}$ is a product measure on $\Omega$).\\
 (ii)There exists a positive integer $K$ such that 
\[ \mathbb{P}\Big{(}\omega_{0\mbox{,}x}=0 \mbox{ for } x \not\in \{-K,-K+1,\dots, K\}\times \{1,2,\dots,K\}\Big{)}=1 .\]
(iii)There exists some $\delta >0$ such that
\begin{equation*}
 \mathbb{P}\Big{(}\omega_{0\mbox{,}x} > \delta \mbox{ for } x \in \lbrace-K,-K+1,\dots, K\rbrace\times\lbrace 1,2,\cdots, K \rbrace \Big{)}=1  .  
\end{equation*}
\end{assumption}
\begin{remark}
 Conditon (iii) is quite a strong condition. The only place where we have used it are in Lemma \ref{expL} and in showing the irreducibility of the random walk $\overline{q}$ in Section 3.2 and the Markov chain $Z_k$ in the proof of Lemma \ref{2sup}. Condition (iii) can certainly be made weaker. 
\end{remark}

Before we proceed, a few words on the notation. For $x\in \mathbb{R}$, $[x]$ will denote the largest integer less than or equal to $x$. For $x,y \in \mathbb{Z}^2$, $P_{x,y}^{\omega}( \cdot) $ will denote the probabilities for two independent walks in the same environment $\omega$ and $P_{x,y}= \mathbb{E} P_{x,y}^{\omega}$.  $\mathbb{E}$ denotes $\mathbb{P}$-expectation. $E_x^{\omega}, E_{x,y}^{\omega}, E_x, E_{x,y}$ are the expectations under $P_x^{\omega}, P_{x,y}^{\omega}, P_x, P_{x,y}$ respectively. For $r,s \in \mathbb{R}$, $P_{r,s}^{\omega}, P_{r,s}, E_{r}^{\omega},E_{r,s}^{\omega},E_{r,s}$ will be shorthands for $P_{([r],0),([s],0)}^{\omega},P_{([r],0),([s],0)},E_{([r],0)}^{\omega},E_{([r],0),([s],0)}^{\omega}$ and $E_{([r],0),([s],0)}$ respectively. $C$ will denote constants whose value may change from line to line. Elements of $\mathbb{R}^2$ are regarded as column vectors. For two vectors $x$ and $y$ in $\mathbb{R}^2$, $x\cdot y$ will denote the dot product between the vectors.

\bigskip

\section{Statement of Result}

Let $\lambda_n = \inf \{k\ge 1:X_k \cdot e_2 \geq n \}$ denote the first time when the walk reaches level $n$ or above. Denote the drift $D(x,\omega)$ at the point $x$ by 
\[ D(x,\omega) = \sum_{z}z \omega_{x,z} \]
and let
\begin{eqnarray}
 \label{w}
 \hat{w}= \Big{(}1, - \frac{\mathbb{E}(D\cdot e_1)}{\mathbb{E}(D\cdot e_2)} \Big{)}^T .
\end{eqnarray}
Let $B(\cdot)$ be a Brownian motion in $\mathbb{R}^2$ with diffusion matrix $\Gamma=\mathbb{E}(DD^{T})-\mathbb{E}(D)\mathbb{E}(D)^{T}$. For a fixed positive integer $N$ and real numbers $r_1,r_2,\cdots,r_N$ and $\theta_1,\theta_2,\cdots,\theta_N$, define the function $h:\mathbb{R} \to \mathbb{R}$ by  
\begin{eqnarray}
\label{h}
h(s)=\sum_{i=1}^N \sum_{j=1}^N \theta_i \theta_j\frac{\sqrt{s}}{\beta \overline{\sigma}^2}\int_0^{\frac{\overline{\sigma}^2}{c_1}} \frac{1}{\sqrt{2\pi v}} \exp\big( -\frac{(r_i-r_j)^2}{2sv} \big) dv
\end{eqnarray}
for positive constants $\beta, \overline{\sigma}, c_1$ defined further below in (\ref{beta}), (\ref{sigma}) and (\ref{c1}) respectively.\\
Let $e_1=(1,0)^T$. Define for $s,r \in \mathbb{R}$
\begin{eqnarray}
 \label{xi} 
 \xi_n(s,r)= E_{r\sqrt{n}}^{\omega}(X_{\lambda_{[ns]}}\cdot e_1)-E_{r\sqrt{n}}(X_{\lambda_{[ns]}}\cdot e_1).
\end{eqnarray}
Notice that for any fixed $r \in \mathbb{R}$, $\xi_n(s,r)$ is the centered quenched mean of the position on crossing level $[ns]$ of a random walk starting at $(r\sqrt{n},0)$. The theorem below gives a functional CLT for $\xi_n(s,r)$.

\bigskip

\begin{theorem}
\label{theorem}
Fix a positive integer $N$. For any $N$ distinct real numbers $r_1<r_2<\cdots < r_N$ and for all vectors $\mathcal{\theta}=(\theta_1,\theta_2,\cdots,\theta_N)$, we have

\[ \sum_{i=1}^N \theta_i \frac{\xi_n(\cdot,r_i) }{n^{\frac{1}{4}}} \Rightarrow \hat{w} \cdot B(h(\cdot))
 \]
where the above convergence is the weak convergence of processes in $D[0,\infty)$ with the Skorohod topology. 
\end{theorem}

\bigskip

\begin{remark} $\Phi(\cdot)= \hat{w} \cdot B(h(\cdot))$ is a mean zero Gaussian process  with covariances
\[ Cov\Big{(}\Phi(s),\Phi(t)\Big{)}=h(s) \hat{w}^T \Gamma \hat{w} , \hspace{1cm} s<t.\]
\end{remark}

\bigskip

\begin{corollary}
\[  \frac{\xi_n(\cdot,0)}{n^{\frac{1}{4}}} \Rightarrow \hat{w}\cdot B(g(\cdot)) \]
where the above convergence is the weak convergence of processes in $D[0,\infty)$ with the Skorohod topology. Here  
\[ g(s)=\frac{2\sqrt{s} }{\beta \overline{\sigma} \sqrt{c_1}\sqrt{2\pi}} \]
for positive constants $\beta, \overline{\sigma}, c_1$ defined in (\ref{beta}), (\ref{sigma}) and (\ref{c1}) respectively.
\end{corollary}

\bigskip

\section{Proof of Theorem 2.1}

We begin with a variation of the well known Martingale Functional Central Limit Theorem whose proof is deferred to the Appendix.
\bigskip
\begin{lemma}
 \label{MCLT}
 Let $ \{X_{n,m}, \mathcal{F}_{n,m}, 1 \leq m \leq n \}$ be an $\mathbb{R}^d$ -valued square integrable martingale difference array on a probability space $(\Omega, \mathcal{F} , P )$. Let $\Gamma$ be a symmetric, non-negative definite $d \times d$ matrix. 
Let $h(s)$ be an increasing $\alpha$-H\"older continuous function on $[0,1]$ with $h(0)=0$ and $h(1)=\gamma >0$.
Define $S_n(s)= \sum_{k=1}^{[ns]}X_{n,k} $. 
Assume that
\begin{eqnarray}
\label{cond1}
\lim \limits_{n \to \infty} \sum_{k=1} ^ {[ns]} E\Big{(}X_{n,k} X_{n,k}^T \Big{\vert} \mathcal{F}_{n,k-1}\Big{)} =   h(s) \Gamma \mbox{   in probability,  } 
\end{eqnarray}
for each $0 \leq s \leq 1$, and 
\begin{eqnarray}
\label{cond2}
\lim \limits_{ n \to \infty} \sum_{k=1}^{n} E \Big{(} |X_{n,k}|^2 \mathbb{I}\{|X_{n,k}| \geq \epsilon\} \Big{\vert} \mathcal{F}_{n,k-1} \Big{)} = 0  \mbox {   in probability,   }  
\end{eqnarray}
for each $\epsilon  > 0$. Then $S_n(\cdot)$ converges weakly to the process $\Xi(\cdot)$ on the space $D[0,1]$ with the Skorohod topology. Here $\Xi(s)=B(h(s))$ where $B(\cdot)$ is a Brownian motion with diffusion matrix $\Gamma$. 
\end{lemma}
\vspace{0.3cm}
 Let $\mathcal{F}_0=\lbrace \emptyset,\Omega \rbrace$ and $\mathcal{F}_k=\sigma\Big{\{}\bar{\omega}_j : j \le k-1\Big{\}}$ where $\bar{\omega}_j = \lbrace \omega_{x \cdot} : x\cdot e_2=j \rbrace$. $\mathcal{F}_k$ thus denotes the part of the environment strictly below level $k$ (level $k$ here denotes all points $\{x: x\cdot e_2=k\}$). Notice that for all $x$ and for each $i \in \{1,2,\cdots,N\}$, $P_{r_i\sqrt{n}}^{\omega}(X_{\lambda_{k}}=x)$ is $\mathcal{F}_k$ measurable and hence also is $E_{r_i \sqrt{n}}^{\omega}(X_{\lambda_{k}})$. Let $\{X \mbox{ hits level }k \}$ be the event $\{X_{\lambda_k} \cdot e_2=k \}$. Now
\begin{eqnarray*}
 \mathbb{E}\Big{[} E_{r_i\sqrt{n}}^{\omega}(X_{\lambda_k})-E_{r_i\sqrt{n}}^{\omega}(X_{\lambda_{k-1}})
   \Big{\vert} \mathcal{F}_{k-1} \Big{]}
&=&\mathbb{E} \Big{[} \sum_{x \cdot e_2 =k-1}D(x,\omega) P_{r_i\sqrt{n}}^{\omega}(X_{\lambda_{k-1}}=x) \Big{\vert} \mathcal{F}_{k-1} \Big{]}\\
&=&\mathbb{E}D \cdot P_{r_i\sqrt{n}}^{\omega}(X \mbox{ hits level } k-1).  
\end{eqnarray*}
Let 
\begin{eqnarray}
\label{mart}
M_k^{n,i}=E_{r_i \sqrt{n}}^{\omega}(X_{\lambda_k}) -([r_i\sqrt{n}],0)^T - \mathbb{E}D \cdot
\sum_{l=1}^k P_{r_i \sqrt{n}}^{\omega}(X \mbox{ hits level } l-1 ) . 
\end{eqnarray}
The above computation tells us $\big{\{}n^{-\frac{1}{4}}\sum_{i=1}^N \theta_i (M_k^{n,i} - M_{k-1}^{n,i}), 1 \leq k \leq n \big{\}} $  is a martingale difference array with respect to $\mathcal{F}_{n,k} = \mathcal{F}_k = \sigma\lbrace\bar  {\omega}_j : j \le k-1\rbrace$. We will now check the conditions of Lemma \ref{MCLT} for $X_{n,k}=n^{-\frac{1}{4}}\sum_{i=1}^N \theta_i (M_k^{n,i} - M_{k-1}^{n,i})$ using the function $h$ in (\ref{h}). The second condition is trivial to check since the difference $M_k^{n,i}-M_{k-1}^{n,i}$ is bounded.

The main work in the paper is checking the condition (\ref{cond1}) in the above lemma. First note that
\begin{eqnarray*}
 M_k^{n,i}-M_{k-1}^{n,i}
&=& E_{r_i\sqrt{n}}^{\omega}( X_{\lambda_k})- E_{r_i\sqrt{n}}^{\omega}( X_{\lambda_{k-1}})- \mathbb{E}D \cdot P_{r_i\sqrt{n}}^{\omega}(X \mbox{ hits level } k-1) \\
&=& \sum_{x \cdot e_2 = k-1} D(x,\omega) P_{r_i\sqrt{n}}^{\omega}(X_{\lambda_{k-1}}=x) -\mathbb{E}D \cdot P_{r_i\sqrt{n}}^{\omega}(X \mbox{ hits level } k-1) \\
&=& \sum_{x \cdot e_2 = k-1} P_{r_i\sqrt{n}}^{\omega} (X_{\lambda_{k-1}} = x) \big{\lbrace}D(x,\omega) - \mathbb{E}D \big{\rbrace}.
\end{eqnarray*}
Using the fact that $\mathbb{E}\Big{[} D(x,\omega)- \mathbb{E}D \Big{]}=0$ and that $D(x,\omega)- \mathbb{E}D$ is independent of $\mathcal{F}_{k-1}$, we get
\[ \sum_{k=1}^{[ns]} \mathbb{E}\Big{[} \Big{\lbrace} \sum_{i=1}^N\theta_i\Big{(}\frac{M_k^{n,i}-M_{k-1}^{n,i}  }{n^{\frac{1}{4}} } \Big{)}\sum_{j=1}^N \theta_j \Big{(} \frac{M_k^{n,j}-M_{k-1}^{n,j} }{n^{\frac{1}{4}} }\Big{)}^T\Big{\rbrace}  \Big{\vert} \mathcal{F}_{k-1}\Big{]}\]
\[ \qquad \qquad = \frac{1}{\sqrt{n}}\sum_{i=1}^N\sum_{j=1}^N \theta_i \theta_j \sum_{k=1}^{[ns]} \mathbb{E} \Big{[}\Big{(} M_k^{n,i}-M_{k-1}^{n,i }  \Big{)}\Big{(} M_k^{n,j}-M_{k-1}^{n,j } \Big{)}^T \Big{\vert} \mathcal{F}_{k-1}\Big{]}\]
\begin{eqnarray}
\label{eqn}
 \hspace{1cm} = \frac{\Gamma}{\sqrt{n}}\sum_{i=1}^N\sum_{j=1}^N \theta_i \theta_j \sum_{k=1}^{[ns]} P^{\omega}_{r_i \sqrt{n},\, r_j \sqrt{n}}
 (X_{\lambda_{k-1}}= \tilde{X}_{\tilde{\lambda}_{k-1}} \mbox{ and both walks hit level }k-1). 
 \end{eqnarray}
Here $\Gamma= \mathbb{E}(DD^T)-\mathbb{E}(D)\mathbb{E}(D)^{T}$ and $X$, $\tilde{X}$ are independent walks in the same environment starting at $([r_i\sqrt{n} ],0)$ and $([r_j \sqrt{n} ],0)$ respectively. We will later show that the above quantity converges in $\mathbb{P}$-probability as $n \rightarrow \infty$ to $h(s)\Gamma $ where $h(s)$ is the function in (\ref{h}). We will also show that $h$ is increasing and H\"older continuous. Lemma \ref{MCLT} thus gives us
\begin{eqnarray}
 \label{star}
 \sum_{i=1}^N \theta_i \frac{ M^{n,i}_{[n\cdot]}}{n^{\frac{1}{4}}} \Rightarrow B(h(\cdot)).
 \end{eqnarray}
 From (\ref{star}) we can complete the proof of Theorem \ref{theorem} as follows. Recalling the definiton in (\ref{xi}),
\[ \xi_n(s,r_i) = E_{r_i\sqrt{n}}^{\omega}(X_{\lambda_{[ns]}}\cdot e_1) - E_{r_i\sqrt{n}}(X_{\lambda_{[ns]}}\cdot e_1)\]
and let
\[ \zeta_n(s,r_i)= E_{r_i\sqrt{n}}^{\omega}(U_{[ns]-1}) - E_{r_i \sqrt{n}}(U_{[ns]-1}) .\]
Here $U_n=\vert \lbrace k : X_k \cdot e_2 \le n \rbrace \vert $ is the number of levels hit by the walk up to level $n$.
Since  from equation (\ref{mart}),
\[E_{r_i\sqrt{n}}(X_{\lambda_{[ns]}}\cdot e_1)= [r_i \sqrt{n}] +\mathbb{E}(D \cdot e_1)E_{r_i \sqrt{n}}(U_{[ns]-1})\]
 we have
\begin{eqnarray*}
M_{[ns]}^{n,i} \cdot e_1 &= &E_{r_i \sqrt{n}}^{\omega}(X_{\lambda_{[ns]}}\cdot e_1) - [r_i \sqrt{n}]- \mathbb{E}(D \cdot e_1) E_{r_i \sqrt{n}}^{\omega}(U_{[ns]-1}) \\
&=& E_{r_i \sqrt{n}}^{\omega}(X_{\lambda_{[ns]}}\cdot e_1) - E_{r_i \sqrt{n}}(X_{\lambda_{[ns]}}\cdot e_1)- \mathbb{E}(D \cdot e_1) \big{[}E_{r_i \sqrt{n}}^{\omega}(U_{[ns]-1}) -E_{r_i \sqrt{n}}(U_{[ns]-1})\big{]} \\
&=&\xi_n(s,r_i) - \mathbb{E}(D \cdot e_1) \zeta_n(s,r_i).
\end{eqnarray*}
$0\le X_{\lambda_{[ns]}}\cdot e_2 - [ns] \le K$ gives us $\vert E_{r_i \sqrt{n}}^{\omega} (X_{\lambda_{[ns]}} \cdot e_2)-E_{r_i \sqrt{n}} (X_{\lambda_{[ns]}} \cdot e_2)\vert \leq K$. Now
\begin{eqnarray*}
M_{[ns]}^{n,i} \cdot e_2&=& E^{\omega}_{r_i \sqrt{n}}(X_{\lambda_{[ns]}} \cdot e_2)-\mathbb{E}(D \cdot e_2) E_{r_i \sqrt{n}}^{\omega}(U_{[ns]-1}) \\
&=& E^{\omega}_{r_i \sqrt{n}}(X_{\lambda_{[ns]}} \cdot e_2)- E_{r_i \sqrt{n}}(X_{\lambda_{[ns]}} \cdot e_2)- \mathbb{E}(D \cdot e_2)\big{[} E_{r_i \sqrt{n}}^{\omega}(U_{[ns]-1})- E_{r_i \sqrt{n}}(U_{[ns]-1}) \big{]} \\
&=& O(1) - \mathbb{E}(D \cdot e_2) \zeta_n(s,r_i).
\end{eqnarray*}
Thus 
\[\zeta_n(s,r_i)= O(1) - \frac{M_{[ns]}^{n,i} \cdot e_2}{\mathbb{E}(D \cdot e_2)} \]
and
\[ \xi_n(s,r_i)= M_{[ns]}^{n,i} \cdot e_1+\mathbb{E}(D \cdot e_1) \zeta_n(s,r_i)= M_{[ns]}^{n,i} \cdot e_1 - \frac{\mathbb{E}(D \cdot e_1)}{\mathbb{E}(D \cdot e_2)} M_{[ns]}^{n,i} \cdot e_2 +O(1). \]
So (recall the definiton of $\hat{w}$ in (\ref{w})),
\begin{eqnarray*}
 \sum_{i=1}^N \theta_i\frac{\xi_{n}(\cdot,r_i) } {n^{\frac{1}{4}}} &=& n^{- \frac{1}{4}} \sum_{i=1}^N \theta_i \Big{[}E_{r_i \sqrt{n}}^{\omega} (X_{\lambda_{[n\cdot]}} \cdot e_1)-E_{r_i \sqrt{n}} (X_{\lambda_{[n\cdot]}} \cdot e_1)\Big{]}\\
 &=& \hat{w} \cdot \sum_{i=1}^N \theta_i  \frac{M_{[n\cdot]}^{n,i}}{n^{\frac{1}{4}}} +O(n^{-\frac{1}{4}}) \\
&\Rightarrow & \hat{w} \cdot B\big{(}h(\cdot)\big{)}.  
 \end{eqnarray*}
Returning to equation (\ref{eqn}), the proof of Theorem \ref{theorem} will be complete if we show 
\[ \frac{1}{\sqrt{n}} \sum_{k=1}^{n}P^{\omega}_{r_i \sqrt{n},\,r_j \sqrt{n}}
 (X_{\lambda_{k-1}}= \tilde{X}_{\tilde{\lambda}_{k-1}} \mbox{ on level }k-1) \longrightarrow 
 \frac{1}{\beta \overline{\sigma}^2}\int_0^{\frac{\overline{\sigma}^2}{c_1}} \frac{1}{\sqrt{2\pi v}} \exp\big( -\frac{(r_i-r_j)^2}{2v} \big) dv \]
 as $ n \to \infty$. We will first show the averaged statement (with $\omega$ integrated away)

\bigskip

\begin{proposition} 
\label{STEP 1}
\[\frac{1}{\sqrt{n}} \sum_{k=1}^{n}P_{r_i \sqrt{n},\,r_j \sqrt{n}}
(X_{\lambda_{k-1}}= \tilde{X}_{\tilde{\lambda}_{k-1}} \mbox{ on level } k-1)\longrightarrow 
\frac{1}{\beta \overline{\sigma}^2}\int_0^{\frac{\overline{\sigma}^2}{c_1}} \frac{1}{\sqrt{2\pi v}} \exp\big( -\frac{(r_i-r_j)^2}{2v} \big) dv\]
\end{proposition}
 and then show that the difference of the two vanishes: 

\bigskip

\begin{proposition}
\label{STEP 2} 
 \[\frac{1}{\sqrt{n}} \sum_{k=1}^{n}P^{\omega}_{r_i \sqrt{n},\,r_j \sqrt{n}}
(X_{\lambda_{k-1}}= \tilde{X}_{\tilde{\lambda}_{k-1}} \mbox{ on level }k-1) \hspace{8cm} \]
\[\qquad \qquad \qquad - \mbox{ } \frac{1}{\sqrt{n}} \sum_{k=1}^{n}P_{r_i \sqrt{n},\,r_j \sqrt{n}}
(X_{\lambda_{k-1}}= \tilde{X}_{\tilde{\lambda}_{k-1}} \mbox{ on level }k-1) \longrightarrow 0 \,\,\, \mbox{in}\,\, \mathbb{P}\mbox{-probability}.\]
\end{proposition} 

\bigskip

\subsection{Proof of Proposition \ref{STEP 1}}
For simplicity of notation, let $P_{(i,j)}^n$ denote $P_{r_i \sqrt{n},r_j \sqrt{n}}$ and $E_{(i,j)}^n$ denote $E_{r_i \sqrt{n},r_j \sqrt{n}}$. Let $0=L_0<L_1< \cdots$ be the successive common levels of the independent walks (in common environment) $X$ and $\tilde{X}$; that is $L_j= \inf \{l>L_{j-1}:X_{\lambda_l}\cdot e_2=\tilde{X}_{\tilde{\lambda}_l} \cdot e_2 =l\}$. Let
\[Y_k=X_{\lambda_{L_k}}\cdot e_1-\tilde{X}_{\tilde{\lambda}_{L_k}}\cdot e_1.\]
 Now
\[
 \frac{1}{\sqrt{n}}  \sum_{k=1}^{n} P_{(i,j)}^n
 (X_{{\lambda}_{k-1}} = \tilde{X}_{\tilde{\lambda}_{k-1}}   \mbox{ on level } k-1 ) \]
\[
\qquad \, \, \, \,=\frac{1}{\sqrt{n}}  \sum_{k=1}^{n} P_{(i,j)}^n
(X_{{\lambda}_{L_{k-1}}} = \tilde{X}_{\tilde{\lambda}_{L_{k-1}} }  \mbox{ and } L_{k-1}
 \leq n-1 )  \]

\begin{eqnarray}
\label{eqn3}
\,\,\,\,=\frac{1}{\sqrt{n}}  \sum_{k=1}^{n} P_{(i,j)}^n(Y_{k-1}=0  \mbox{ and } L_{k-1} \leq n-1 ).
\end{eqnarray}

We would like to get rid of the inequality $L_{k-1} \le n-1$ in the expression above so that we have something resembling a Green's function. Denote by $\Delta L_j=L_{j+1}-L_j$. Call $L_k$ a meeting level (m.l.) if the two walks meet at that level, that is $X_{\lambda_{L_k}}=\tilde{X}_{\tilde{\lambda}_{L_k}}$. Also let $I=\{j:L_j \mbox{ is a meeting level } \}$. Let $Q_1,Q_2,\cdots $ be the consecutive $\Delta L_j$ where $ j \in I$ and $R_1,R_2,\cdots $ be the consecutive $\Delta L_j$ where $ j \notin I$. We start with a simple observation.

\bigskip

\begin{lemma}
\label{iid}
Fix $x,y \in \mathbb{Z}$. Under $P_{x,y}$,
\[ Q_1,Q_2,\cdots \mbox{ are i.i.d.  with common distribution } P_{(0,0)(0,0)}(L_1 \in \cdot) \]
\[ R_1,R_2, \cdots \mbox{ are i.i.d. with common distribution } P_{(0,0),(1,0)}(L_1 \in \cdot) \]
\end{lemma}
\begin{proof}
 We prove the second statement. The first statement can be proved similarly. Call a level a non meeting common level (n.m.c.l) if the two walks hit the level but do not meet at that level.
For positive integers $k_1,k_2,\cdots,k_n$,
\[P_{x,y}\big(R_1=k_1,R_2=k_2,\cdots, R_n=k_n\big)= \sum_i P_{x,y}\big(R_1=k_1, \cdots, R_{n}=k_{n}, i \mbox{ is the } n \mbox{'th  n.m.c.l.}\big)\]
\[= \sum_{i,j\ne 0} P_{x,y}\big(R_1=k_1, \cdots, R_{n-1}=k_{n-1}, i \mbox{ is the } n \mbox{'th  n.m.c.l.},\tilde{X}_{\tilde{\lambda}_{i}}\cdot e_1-X_{\lambda_{i}}\cdot e_1=j,R_n=k_n\big)\]
\[=\sum_{i,j\ne 0,l} \mathbb{E}\Big{[} P_{x,y}^{\omega}\big(R_1=k_1, \cdots, R_{n-1}=k_{n-1}, i \mbox{ is the } n \mbox{'th  n.m.c.l.},X_{\lambda_{i}}\cdot e_1=l,\tilde{X}_{\tilde{\lambda}_{i}}\cdot e_1=j+l\big)\]
\[\hspace{5cm} \times P_{(l,i),(l+j,i)}^{\omega}(\Delta L_0 = k_n) \Big{]} \]
 \[=\sum_{i,j\ne 0,l} \mathbb{E}\Big{[} P_{x,y}^{\omega}\big(R_1=k_1, \cdots, R_{n-1}=k_{n-1}, i \mbox{ is the } n \mbox{'th  n.m.c.l.},X_{\lambda_{i}}\cdot e_1=l,\tilde{X}_{\tilde{\lambda}_{i}}\cdot e_1=j+l\big) \Big{]} \]
\[\hspace{5cm} \times P_{(0,0),e_1}( L_1 = k_n)  \]
 \[= P_{(0,0),e_1}(L_1=k_n)P_{x,y}(R_1=k_1,R_2=k_2,\cdots,R_{n-1}=k_{n-1}).\hspace{5cm} \]
The fourth equality is because $P^{\omega}_{(l,i),(l+j,i)}(\Delta L_0=k_n)$ depends on $\overline{\omega}_i$ whereas the previous term depends on the part of the environment strictly below level $i$. The proof is complete by induction.
\end{proof}

\bigskip

\begin{lemma}
\label{expL}
There exists some $a>0$ such that 
\[ E_{(0,0),(0,0)}(e^{aL_1}) < \infty \hspace{0.5cm} \mbox{ and } \hspace{0.5cm}
  E_{(0,0),(1,0)}(e^{aL_1}) < \infty.
\]

\end{lemma}
\begin{proof} 
By the ellipticity assumption (iii), we have
\[ P_{(0,0),(0,0)}(L_1 > k ) \le (1-\delta)^{[\frac{k}{K} ]}, \]
\[P_{(0,0),(1,0)} (L_1 > k ) \le (1-\delta)^{[ \frac{k}{K} ]}. \]
Since $L_1$ is stochastically dominated by a geometric random variable, we are done. 
\end{proof}
\vspace{0.3cm}

Let us denote by $X_{[0,n]}$ the set of points visited by the walk upto time $n$. It has been proved in Proposition 5.1 of \cite{rsqi} that for any starting points $x,y \in \mathbb{Z}^2$
\begin{eqnarray}
\label{intup}
 E_{x,y} \Big{(} \vert X_{[0,n]} \cap \tilde{X}_{[0,n]} \vert \Big{)} \leq C\sqrt{n}. 
 \end{eqnarray}
This inequality is obtained by control on a Green's function. The above lemma and the inequality that follows tell us that common levels occur very frequently but the walks meet rarely. Let
\begin{eqnarray}
\label{c1}
 E_{(0,0),(1,0)}(L_1) = c_1 , \qquad E_{(0,0),(0,0)}(L_1)=c_0 .
 \end{eqnarray}
We will need the following lemma.

\bigskip

\begin{lemma}
\label{LDP}
For each $\epsilon >0$, there exist constants $C>0 \mbox{, }b(\epsilon)>0\mbox{, } d(\epsilon)>0$
such that 
\[ P_{(i,j)}^n\Big{(}\frac{L_n}{n} \ge c_1+ \epsilon \Big{)} \le C\exp\big{[}-nb(\epsilon) \big{]}, \]
\[P_{(i,j)}^n\Big{(}\frac{L_n}{n} \le c_1- \epsilon \Big{)} \le C\exp\big{[}-nd(\epsilon) \big{]} .\]
Thus $\frac{L_n}{n} \rightarrow c_1 \,\, P_{0,0}$ a.s.
\end{lemma}

\begin{proof}
 We prove the first inequality. From Lemma \ref{expL}, we can find $a>0$ and some $\nu>0$ such that for each $n$,
 \[ E_{(i,j)}^n\big{(} \exp (a L_n)\big{)} \le \nu^n. \]
 We thus have 
 \begin{eqnarray*}
 P_{(i,j)}^n(\frac{L_n}{n} \ge C_1) & \le& \frac{E_{(i,j)}^n\big{(}\exp (a L_n) \big{)}}{ \exp(aC_1n)}\\
 &\le  & \exp \{ n(\log \nu - aC_1)\}.
 \end{eqnarray*}
 Choose $C_1 >>0$ so that $\log \nu -aC_1 <0$. Now  
\begin{eqnarray*}
\label{1}
P_{(i,j)}^n(\frac{L_n}{n} \ge c_1+\epsilon) \le \exp\{n(\log \nu - a C_1)\} +P_{(i,j)}^n(c_1+\epsilon \le \frac{L_n}{n} \le C_1). 
\end{eqnarray*}
Let $\gamma=\frac{\epsilon}{4c_0}$. Denote by $I_n = \{ j: 0\le j < n, L_j \mbox{ is a meeting level } \}$ and recall $\Delta L_j =L_{j+1}-L_j$. We have 
\begin{eqnarray*}
\label{2}
P_{(i,j)}^n( c_1 +\epsilon \le \frac{L_n}{n} \le C_1 ) &\le& P_{(i,j)}^n \big(  \vert I_n \vert \ge \gamma n,L_n \le C_1n \big) \\
&\;+&P_{(i,j)}^n\Big( \vert I_n \vert < \gamma n, c_1+\epsilon \le \frac{ \sum_{ j \notin I_n, j <n } \Delta L_j }{n} + \frac{ \sum_{ j \in I_n } \Delta L_j }{n} \le C_1 \Big).
\end{eqnarray*}
Let $T_1,T_2,\dots$ be the increments of the successive meeting levels. By an argument like the one given in Lemma \ref{iid}, $\{T_i \}_{i \ge 1} $ are i.i.d. Also from (\ref{intup}), it follows that $E(T_1)=\infty$. Hence we can find $M=M(\gamma) >>2\frac{C_1}{\gamma}$ and $K=K(\gamma)$ such that 
\[E(T_1\wedge K)\ge M.\]
Now,
\begin{eqnarray*}
P_{(i,j)}^n \big( \vert I_n \vert \ge \gamma n , L_n \le C_1 n \big) &\le& P_{(i,j)}^n(T_1 +T_2+\cdots+T_{[\gamma n]} \le C_1n) \\
&\le& P_{(i,j)}^n \Big( \frac{T_1 +T_2 +\cdots T_{[\gamma n]} }{[\gamma n]} \le 2\frac{C_1}{\gamma} \Big)\\
&\le& P_{(i,j)}^n \Big( \frac{T_1 \wedge K +T_2\wedge K +\cdots T_{[\gamma n]} \wedge K }{[\gamma n]} \le 2\frac{C_1}{\gamma} \Big)\\
&\le&  \exp[-nb_2].
\end{eqnarray*} 
for some $b_2 >0$. The last inequality follows from standard large deviation theory. Also

\[P_{(i,j)}^n\Big( \vert I_n \vert < \gamma n, c_1+\epsilon \le \frac{ \sum_{ j \notin I_n,j <n } \Delta L_j }{n} + \frac{ \sum_{ j \in I_n } \Delta L_j }{n} \le C_1 \Big) \hspace{5cm}\]
\[ \le P_{(i,j)}^n\Big( \vert I_n \vert < \gamma n, c_1+\frac{\epsilon}{2} \le \frac{1}{n} \sum_{j \notin I_n, j<n} \Delta L_j \Big) +  P_{(i,j)}^n\Big( \vert I_n \vert < \gamma n,  \frac{1}{n} \sum_{j \in I_n} \Delta L_j \ge \frac{\epsilon}{2} \Big).\]
Let $\{M_j\}_{j\ge 1}$ be i.i.d. with the distribution of $L_1$ under $P_{0,e_1}$ and $\{N_j\}_{j\ge 1}$ be i.i.d. with the distribution of $L_1$ under $P_{0,0}$. We thus have that the above expression is less than
\begin{eqnarray*}
  P\Big( \frac{1}{n}\sum_{j=1}^n M_j \ge c_1 +\frac{\epsilon}{2} \Big) +P\Big( \frac{1}{n} \sum_{j=1}^{[\gamma n]} N_j \ge \frac{\epsilon}{2} \Big) & \le & \exp(-nb_3(\epsilon)) + P\Big( \frac{1}{[\gamma n]} \sum_{j=1}^{[\gamma n]} N_j \ge \frac{\epsilon}{2\gamma} \Big) \\
 &\le& \exp(-nb_3(\epsilon))+\exp(-nb_4(\epsilon)).
 \end{eqnarray*}
for some $b_3(\epsilon), b_4(\epsilon)>0$. Recall that $\frac{\epsilon}{2\gamma} > c_0$ by our choice of $\gamma$. Combining all the inequalities, we have 
\[P_{(i,j)}^n\big( \frac{L_n}{n} \ge c_1 + \epsilon \big) \le 3 \exp(-nb(\epsilon)). \]
for some $b(\epsilon)>0$. The proof of the second inequality is similar. 
\end{proof}

Returning to (\ref{eqn3}), let us separate the sum into two parts as
 \[\frac{1}{\sqrt{n}}  \sum_{k=0}^{[\frac{n(1+\epsilon)}{c_1}]} P_{(i,j)}^n
 (Y_k=0   \mbox{ and } L_k \leq n-1 ) \hspace{5cm}\]
\[ \hspace{5cm} + \frac{1}{\sqrt{n}}  \sum_{k=[\frac{n(1+\epsilon)}{c_1}]+1}^{n-1} P_{(i,j)}^n
 (Y_k=0  \mbox{ and } L_k \leq n-1 ).\hspace{0.5cm}\]
Now the second term  above \; 
is
\[ \frac{1}{\sqrt{n}}  \sum_{k=[\frac {n(1+\epsilon)}{c_1}]+1}^{n-1} P_{(i,j)}^n (X_{\lambda_{L_{k}}} = \tilde{X}_{\tilde{\lambda}_{L_{k}}}   \mbox{ and } L_k \leq n-1 )  \leq  \frac{Cn}{\sqrt{n}}P_{(i,j)}^n(L_{[\frac{n(1+\epsilon)}{c_1}]} \leq n )\]
which goes to $0$ as $n$ tends to infinity by Lemma \ref{LDP}. Similarly
\begin{eqnarray*}
 \frac{1}{\sqrt{n}} \Big{[}\sum_{k=0}^{[\frac{n(1-\epsilon)}{c_1}] } P_{(i,j)}^n(Y_k=0) - \sum_{k=0}^{[\frac{n(1-\epsilon)}{c_1}]}  P_{(i,j)}^n(Y_k=0, L_k \leq n-1 ) \Big{]} 
&\leq& \frac{1}{\sqrt{n}} \sum_{k=0}^{[\frac{n(1- \epsilon)}{c_1} ]} P_{(i,j)}^n(L_k \geq n) \\
&\leq& \frac{Cn}{\sqrt{n}} P_{(i,j)}^n\big{(}L_{[\frac{n(1-\epsilon)}{c_1}]} \geq n \big{)} 
\end{eqnarray*}
also goes to $0$ as $n$ tends to infinity. Thus
\[ \frac{1}{\sqrt{n}} \sum_{k=0}^{n-1} P_{(i,j)}^n(Y_k=0, L_k \leq n-1) \hspace{8cm}\]  
\[\qquad  = \frac{1}{\sqrt{n}} \sum_{k=0}^{[\frac{n(1-\epsilon)}{c_1}]} P_{(i,j)}^n (Y_k=0) + \frac{1}{\sqrt{n}} \sum_{k=[\frac{n(1-\epsilon)}{c_1}]+1}^{[\frac{n(1+\epsilon)}{c_1}]} P_{(i,j)}^n(Y_k=0, L_k \leq n-1) +a_n(\epsilon)\]
where $a_n(\epsilon) \to 0$ as $n \to \infty$. Now we will show the second term in 
in the right hand side of the above equation is negligible.
Let $\tau = \min \lbrace j \geq [\frac{n(1-\epsilon)}{c_1}]+1 : Y_j=0 \rbrace$. Using the Markov property for the second line below, we get
\begin{eqnarray*}
 \frac{1}{\sqrt{n}} \sum_{k=[\frac{n(1-\epsilon)}{c_1}]+1}^{[\frac{n(1+\epsilon)}{c_1}]} P_{(i,j)}^n(Y_k=0, L_k \leq n-1) 
&=&\frac{1}{\sqrt{n}} E_{(i,j)}^n \Big{[} \mathbb{I}_{\lbrace L_{\tau} \leq n \rbrace} \sum_{k=\tau}^{[\frac{n(1+\epsilon)}{c_1}]} \mathbb{I}_{\lbrace Y_k=0,L_k\leq n-1 \rbrace} \Big{]} \\
&\le&  \frac{1}{\sqrt{n}} E_{0,0} \Big{[}\sum_{k=0}^{[\frac{2n\epsilon}{c_1}]} \mathbb{I}_{\lbrace Y_k=0 \rbrace} \Big{]} \\
&=& \frac{1}{\sqrt{n}}E_{0,0}\Big{[} \mathbb{I}_{\{L_{[\frac{2n\epsilon}{c_1}]} \leq 4n\epsilon \}} \sum_{k=0}^{[\frac{2n\epsilon}{c_1}]} \mathbb{I}_{\{Y_k=0\}} \Big{]}\\
 &\hspace{1cm}&+ \frac{1}{\sqrt{n}}E_{0,0} \Big{[} \mathbb{I}_{\{L_{[\frac{2n\epsilon}{c_1}]} > 4n \epsilon \}} \sum_{k=0}^{[\frac{2n\epsilon}{c_1}]} \mathbb{I}_{\{Y_k=0\}} \Big{]} .
\end{eqnarray*}
In the expression after the last equality, we have 
\begin{eqnarray*}
\mbox{  First term } &\leq& \frac{1}{\sqrt{n}}  E_{0,0}\Big{(}\big{\vert}X_{[0,4n\epsilon] }\bigcap \tilde{X}_{[0,4n\epsilon]}\big{\vert} \Big{)} \\
&\leq& \frac{C}{\sqrt{n}} \sqrt{ 4n\epsilon } \leq C \sqrt{\epsilon}. 
\end{eqnarray*}
\begin{eqnarray*}
 \mbox{ Second term } 
&\leq& \frac{Cn}{\sqrt{n}} P_{0,0}\big{(}L_{[\frac{2n\epsilon}{c_1}]} > 4n\epsilon \big{)} \longrightarrow 0 .
\end{eqnarray*}
by Lemma \ref{LDP}. This gives us
\[ \frac{1}{\sqrt{n}} \sum_{k=0}^{n-1} P_{(i,j)}^n(Y_k=0, L_k \leq n-1)  = \frac{1}{\sqrt{n}}\sum_{k=0}^{[\frac{n(1-\epsilon)}{c_1}]} P_{(i,j)}^n(Y_k=0 ) + O(\sqrt{\epsilon}) + b_n(\epsilon) \]
where $b_n(\epsilon) \to 0 $ as $n \to \infty$. By the Markov property again and arguments similar to above
\begin{eqnarray*}
\frac{1}{\sqrt{n}}\sum_{k=\big{[}\frac{n(1-\epsilon)}{c_1} \big{]} +1} ^{\big{[}\frac{n}{c_1} \big{]} } P_{(i,j)}^n (Y_k=0) &\le& \frac{1}{\sqrt{n}} \sum_{k=0}^{\big{[}\frac{n\epsilon}{c_1}\big{]}}
P_{0,0}(Y_k=0) \\
&\le& O(\sqrt{\epsilon}) + c_n(\epsilon)
\end{eqnarray*}
where $c_n(\epsilon) \to 0$ as $ n \to \infty$. So what we finally have is
\begin{eqnarray}
\label{eqn7}
 \frac{1}{\sqrt{n}} \sum_{k=0}^{n-1} P_{(i,j)}^n(Y_k=0, L_k \leq n-1 ) =\frac{1}{\sqrt{n}}\sum_{k=0}^{[\frac{n}{c_1}]} P_{(i,j)}^n(Y_k=0 )+ O(\sqrt{\epsilon})+d_n(\epsilon)
 \end{eqnarray}
where $d_n(\epsilon) \to 0$ as $ n \to \infty$.
 
\bigskip

\subsection{Control on the Green's function}
We follow the approach used in \cite{brs} to find the limit as $n \to \infty$ of 
\[\frac{1}{\sqrt{n}}\sum_{k=0}^{[\frac{n}{c_1}]} P_{(i,j)}^n(Y_k=0 )\]
in the right hand side of (\ref{eqn7}). Since $\epsilon>0$ is arbitrary, this in turn will give us the limit of the left hand side of (\ref{eqn7}). 

In the averaged sense $Y_k$ is a random walk on $\mathbb{Z}$ perturbed at $0$ with transition kernel $q$ given by
\[ q(0,y)= P_{(0,0),(0,0)} (X_{\lambda_{L_1}}\cdot e_1 - \tilde{X}_{\tilde{\lambda}_{L_1}}\cdot e_1  = y)  \]
 \[\qquad \qquad \qquad \qquad \, \, q(x,y)=P_{(0,0),(1,0)} (X_{\lambda_{L_1}}\cdot e_1 - \tilde{X}_{\tilde{\lambda}_{L_1}}\cdot e_1  = y-x-1) \hspace{1cm} \mbox{ for } x \neq 0.\]
Denote the transition kernel of the corresponding unperturbed walk by $\overline{q}$. \[\overline{q}(x,y)=P_{(0,0)}\times P_{(0,0)}(X_{\lambda_{L_1}}\cdot e_1 - \tilde{X}_{\tilde{\lambda}_{L_1}}\cdot e_1= y-x). \]
where $P_{(0,0)}\times P_{(0,0)}$ is the measure under which the walks are independent in independent environments. Note that 
\[ q(x,y)=\overline{q}(x,y) \hspace{1cm} \mbox{for } x \neq 0. \]
The $\overline{q}$ walk is easily seen to be aperiodic(from assumption \ref{model} (iii)), irreducible and symmetric and these properties can be transferred to the $q$ walk. The $\overline{q}$ can be shown to have finite first moment (because $L_1$ has exponential moments) with mean $0$. Green functions for the $\overline{q}$ and $q$ walks are given by
\[ \overline{G}_n(x,y)=\sum_{k=0}^{n}\overline{q}^k(x,y) \mbox{ and } G_n(x,y)=\sum_{k=0}^{n}q^k(x,y). \]
The potential kernel $\overline{a}$ of the $\overline{q}$ walk is
 \[ \overline{a}(x) = \lim \limits_{n \to \infty} \{ \overline{G}_n(0,0)-\overline{G}_n(x,0) \} .\]
It is a well known result from Spitzer~\cite{spitzer}(sections $28$ and $29$) that  $\lim \limits_{ x \to \pm \infty} \frac{\overline{a}(x)}{|x|}= \frac{1}{\overline{\sigma}^2}$ where
\begin{eqnarray}
\label{sigma}
\overline{\sigma}^2 = \mbox{ variance of the } \overline{q} \mbox{ walk}.
\end{eqnarray}
Furthermore we can show that (see \cite{brs} page 518)
\begin{eqnarray}
\label{beta}
 \frac{1}{\sqrt{n}} G_{n-1}(0,0) \beta = \frac{1}{\sqrt{n}} E_{0,0}[\overline{a}(Y_n)] \mbox{ where } \beta=\sum_{x \in \mathbb{Z} } q(0,x) \overline{a}(x). 
 \end{eqnarray}
First we will show $ \frac{1}{\sqrt{n}} E_{0,0}|Y_n|$  converges to conclude that $ \frac{1}{\sqrt{n}} E_{0,0}[\overline{a}(Y_n)]$  converges. Notice that $Y_k=X_{\lambda_{L_k}}\cdot e_1-\tilde{X}_{\tilde{\lambda}_{L_k}}\cdot e_1$ is a martingale w.r.t. 
$\{\mathcal{G}_k=\sigma(X_1,X_2,\dots,X_{\lambda_{L_k}},\tilde{X}_1,\tilde{X}_2,\dots,\tilde{X}_{\tilde{\lambda}_{L_k}})  \}$ under the measure $P_{0,0}$. We will use the martingale central limit theorem (\cite{durr} page 414) to show that $\frac{Y_n}{\sqrt{n}}$ converges to a centered Gaussian.
We first show
\begin{eqnarray}
\label{3.13}
\frac{1}{n} \sum_{k=1}^{n} E_{0,0}\Big{(}(Y_k-Y_{k-1})^2 \mathbb{I}_{\lbrace|Y_k- Y_{k-1}|>\epsilon \sqrt{n}\rbrace}\Big{\vert} \mathcal{G}_{k-1}\Big{)} &\rightarrow& 0 \mbox{ in probability}.
\end{eqnarray}
We already have from Lemma \ref{expL} that for some $a>0$
\[ E_{0,0}(e^{\frac{a}{4K}Y_1}) \leq E_{0,0}(e^{a L_1}) < \infty ,\]
\[ E_{(0,0),(1,0)}(e^{\frac{a}{4K}Y_1}) \leq E_{(0,0),(1,0)}(e^{a L_1})  < \infty .\]
Thus $E_{0,0}\big{[} (Y_k - Y_{k-1})^5 \vert \mathcal{G}_{k-1} \big{]} \leq C$ and so
(\ref{3.13}) holds. Now we check 
\[ \frac{1}{n} \sum_{k=1}^{n} E_{0,0}\Big{(}(Y_k-Y_{k-1})^2\vert \mathcal{G}_{k-1}\Big{)} \rightarrow \overline{\sigma}^2 \mbox{ in probability}.\]
Note that
\begin{eqnarray*}
 E_{0,0}\big{[}(Y_k-Y_{k-1})^2 \big{\vert} \mathcal{G}_{k-1}\big{]} &=& E_{0,0}\big{[}(Y_k-Y_{k-1})^2 \mathbb{I}_{\{Y_{k-1} = 0\}} \vert \mathcal{G}_{k-1}\big{]} +  E_{0,0}\big{[}(Y_k-Y_{k-1})^2\mathbb{I}_{\{Y_{k-1} \ne 0\}} \vert\mathcal{G}_{k-1}\big{]}\\
 &=& u_0\mathbb{I}_{\{Y_{k-1}=0\}}+\overline{\sigma}^2\mathbb{I}_{\{Y_{k-1}\neq 0\}}
\end{eqnarray*}
where $u_0=E_{0,0}(Y_1^2)$ and $E_{(1,0),(0,0)}((Y_1-1)^2)=\overline{\sigma}^2$ (as defined in (\ref{sigma})), the variance of the unperturbed walk $\overline{q}$. So
\[ \frac{1}{n} \sum_{k=1}^{n} E_{0,0}\big{[}(Y_k-Y_{k-1})^2\vert\mathcal{G}_{k-1}\big{]} = \overline{\sigma}^2 +\frac{(u_0-\overline{\sigma}^2)}{n} \sum_{k=1}^{n} \mathbb{I}_{\{Y_{k-1}=0\}}. \]
To complete, by choosing $b$ large enough we get
\begin{eqnarray*}
E_{0,0}[\sum_{k=1}^n \mathbb{I}_{\{ Y_{k-1}=0\}} ] &\le& n P_{0,0}(L_n >bn)+E_{0,0}\vert X_{[0,nb]} \cap \tilde{X}_{[0,nb]} \vert\\
&\le& C \sqrt{n}.
\end{eqnarray*} 
Hence $\frac{1}{n} \sum_{k=1}^n \mathbb{I}_{\lbrace Y_{k-1}=0 \rbrace} \to 0$ in $P_{0,0}$-probability. We have checked both the conditions of the martingale central limit theorem and so we have $n^{-\frac{1}{2}}Y_n \Rightarrow N(0,\overline{\sigma}^2)$. 
We now show $E_{0,0} \vert n^{-\frac{1}{2}} Y_n \vert \rightarrow E\vert N(0,\overline{\sigma}^2) \vert = \frac{2 \overline{\sigma}}{\sqrt{2\pi}} $. This will follow if we can show that $n^{-\frac{1}{2}}Y_n$ uniformly integrable. But that is clear since we have
\[ \frac{E_{0,0}Y_n^2}{n} = \overline{\sigma}^2+\frac{u_0-\overline{\sigma}^2}{n} \sum_{k=1}^n P_{0,0}(Y_{k-1}=0) \leq u_0 +\overline{\sigma}^2.\]
It is easily shown that 
\[ \lim \limits_{n \to \infty}\frac{E_{0,0}[\overline{a}(Y_n)]}{\sqrt{n}} = \lim \limits_{n \to \infty} \frac{1}{\overline{\sigma}^2}\frac{E_{0,0}\vert Y_n\vert}{\sqrt{n}} = \frac{2}{\overline{\sigma}\sqrt{2 \pi}}. \] 
We already know by the local limit theorem (\cite{durr} section 2.6) that $\lim \limits_{n\to \infty}\frac{1}{\sqrt{n}} \overline{G}_n(0,0)=\frac{2}{\overline{\sigma}\sqrt{2\pi}}$ which is equal to $\lim \limits_{n \to \infty}n^{-\frac{1}{2}}E_{0,0}[\overline{a}(Y_n)]=\lim \limits_{n\to \infty}n^{-\frac{1}{2}} \beta G_n(0,0)$ by the above computations. The arguments in (\cite{brs} page 518 (4.9)) allow us to conclude that 
\[ \sup_x \frac{1}{\sqrt{n}}\vert \beta G_n(x,0) - \overline{G}_n(x,0) \vert  \to 0.\] 
Now returning back to (\ref{eqn7}), the local limit theorem (\cite{durr} section 2.6) and a Riemann sum argument gives us 
\[ \lim \limits_{n \to \infty} \frac{1}{\sqrt{n}} \overline{G}_{[\frac{n}{c_1}]}([r_i\sqrt{n}]-[r_j \sqrt{n}],0)
=\frac{1}{\overline{\sigma}^2}\int_0^{\frac{\overline{\sigma}^2}{c_1}} \frac{1}{\sqrt{2\pi v}} \exp\big( -\frac{(r_i-r_j)^2}{2v} \big) dv.
\]
Hence the right hand side of (\ref{eqn7}) tends to  
\begin{eqnarray}
\label{green}
\lim \limits_{n \to \infty} \frac{1}{\sqrt{n}} G_{[\frac{n}{c_1}]}( [r_i\sqrt{n}]-[r_j\sqrt{n}],0) 
=\frac{1}{\beta \overline{\sigma}^2}\int_0^{\frac{\overline{\sigma}^2}{c_1}} \frac{1}{\sqrt{2\pi v}} \exp\big( -\frac{(r_i-r_j)^2}{2v} \big) dv.
\end{eqnarray}
This completes the proof of Proposition \ref{STEP 1}\qed.

\bigskip

\subsection{Proof of Proposition \ref{STEP 2}}
This section is based on the proof of Theorem 4.1 in Section of \cite{ff}  and (5.20) in \cite{brs}. Recall $\mathcal{F}_0 = \{ \emptyset , \Omega\}$, $\mathcal{F}_k = \sigma\big{\{} \overline{\omega}_j : j\le k-1 \big{\}}$.
Now
\[ \frac{1}{\sqrt{n}} \sum_{k=0}^{n} \Big{\lbrace}P_{r_i \sqrt{n},r_j \sqrt{n}}^{\omega}(Y_k=0,L_k \leq n)-P_{r_i \sqrt{n},r_j \sqrt{n}}(Y_k=0,L_k \leq n)\Big{\rbrace}  \]
\[ \qquad \qquad =\frac{1}{\sqrt{n}}\sum_{l=0}^{n-1} \sum_{k=l+1}^{n} \Big{\lbrace} P_{r_i\sqrt{n},r_j \sqrt{n}}\big{(}X_{\lambda_k}= \tilde{X}_{\tilde{\lambda}_k} \mbox{ and both walks hit level } k \big{\vert} \mathcal{F}_{l+1} \big{)} \]
\[\qquad \qquad \qquad \qquad \qquad \qquad \quad \, \, -P_{r_i \sqrt{n},r_j \sqrt{n}}\big{(}X_{\lambda_k}= \tilde{X}_{\tilde{\lambda}_k} \mbox{ and both walks hit level } k \big{\vert} \mathcal{F}_{l}\big{)} \Big{\rbrace}. \]
\vspace{0.2cm}
Call $R_l= \sum_{k=l+1}^n \Big{\lbrace} \cdots \Big{\rbrace}$. It is enough to show that $\mathbb{E} (\frac{1}{\sqrt{n}}\sum_{l=0}^{n-1}  R_l)^2  \to 0$. By orthogonality of martingale increments, $\mathbb{E}R_lR_m =0 \mbox{ for } l \neq m$. Let 
 \[ \phi_n= \vert\{k: k \le n ,k \mbox{ is a meeting level } \} \vert ,\]
the number of levels at which the two walks meet up to level $n$.

\bigskip

\begin{proposition}
\label{prop1}
\[\frac{1}{n} \sum_{l=0}^{n} \mathbb{E} R_l^2  \longrightarrow 0 .\]
\end{proposition}
\begin{proof}  
 Notice that $R_l$ is $0$ unless one of the walks hit level $l$ because otherwise the event in question does not need $\overline{\omega}_{l}$. We then have $R_l=R_{l,1}+R_{l,2} + R_{l,2}^{'}$ where
\[ R_{l,1} = \sum_{u\cdot e_2 =l} \hspace{0.2cm}\sum_{\tilde{u} \cdot e_2 =l} P_{r_i \sqrt{n}}^{\omega}(X_{\lambda_l}=u)P_{r_j \sqrt{n}}^{\omega}(\tilde{X}_{\tilde{\lambda}_l}=\tilde{u})  \hspace{5cm}\] \[ \qquad \times \sum_{\begin{tiny}\begin{array}{c}1\le w\cdot e_2\mbox{, } z\cdot e_2  \leq K \\
 -K \le w\cdot e_1\mbox{, } z\cdot e_1  \leq K \end{array}\end{tiny} } E_{u+w,\tilde{u}+z} (\phi_n)  \cdot \Big{\lbrace}\omega(u,w)\omega(\tilde{u},z) -\mathbb{E}[\omega(u,w)\omega(\tilde{u},z)]\Big{\rbrace}  \]
\[ R_{l,2} = \sum_{u\cdot e_2 =l}\hspace{0.2cm} \sum_{\tilde{u} \cdot e_2 <l} \hspace{0.2cm} \sum_{\begin{tiny}\begin{array}{c}l < u_1\cdot e_2 \leq l+K,\\ |(u_1-u)\cdot e_1| \leq K\end{array}\end{tiny}}   \sum_{\begin{tiny} \begin{array}{c}l < \tilde{u}_1\cdot e_2 \leq l+K,\\ |(\tilde{u}_1-\tilde{u})\cdot e_1| \leq K \end{array} \end{tiny}}  P_{r_i \sqrt{n}}^{\omega}(X_{\lambda_l} =u)  \hspace{5cm}\]
\[\qquad \qquad \qquad  \qquad  \times  P_{r_j \sqrt{n}}^{\omega}(\tilde{X}_{\tilde{\lambda}_l}=\tilde{u}_1, \tilde{X}_{\tilde{\lambda}_l -1}=\tilde{u} )  \Big{\lbrace}\omega(u,u_1-u)-\mathbb{E}[\omega(u,u_1-u)] \Big{\rbrace} \cdot E_{u_1,\tilde{u}_1}(\phi_n), \]
and
\[ R_{l,2}^{'} = \sum_{u\cdot e_2 <l}\hspace{0.2cm} \sum_{\tilde{u} \cdot e_2 =l} \hspace{0.2cm} \sum_{\begin{tiny}\begin{array}{c}l < u_1\cdot e_2 \leq l+K,\\ |(u_1-u)\cdot e_1| \leq K\end{array}\end{tiny}}   \sum_{\begin{tiny} \begin{array}{c}l < \tilde{u}_1\cdot e_2 \leq l+K,\\ |(\tilde{u}_1-\tilde{u})\cdot e_1| \leq K \end{array} \end{tiny}}  P_{r_i \sqrt{n}}^{\omega}(X_{\lambda_{l}-1} =u, X_{\lambda_l}=u_1)  \hspace{5cm}\]
\[\qquad \qquad \qquad  \qquad  \times  P_{r_j \sqrt{n}}^{\omega}(\tilde{X}_{\tilde{\lambda}_l}=\tilde{u})  \Big{\lbrace}\omega(\tilde{u},\tilde{u}_1-\tilde{u})-\mathbb{E}[\omega(\tilde{u},\tilde{u}_1-\tilde{u})] \Big{\rbrace} E_{u_1,\tilde{u}_1}(\phi_n). \]
We first work with $R_{l,1}$. Let us order the $z$'s as shown in the Figure $1$ below. Here $\beta \leq z$ if $\beta$ occurs before $z$ in the sequence and $z+1$ is the next element in the sequence. Also, without loss of generality, we have assumed $(K,K)$ is the last $z$.\\

\begin{center} \includegraphics[height=2in]{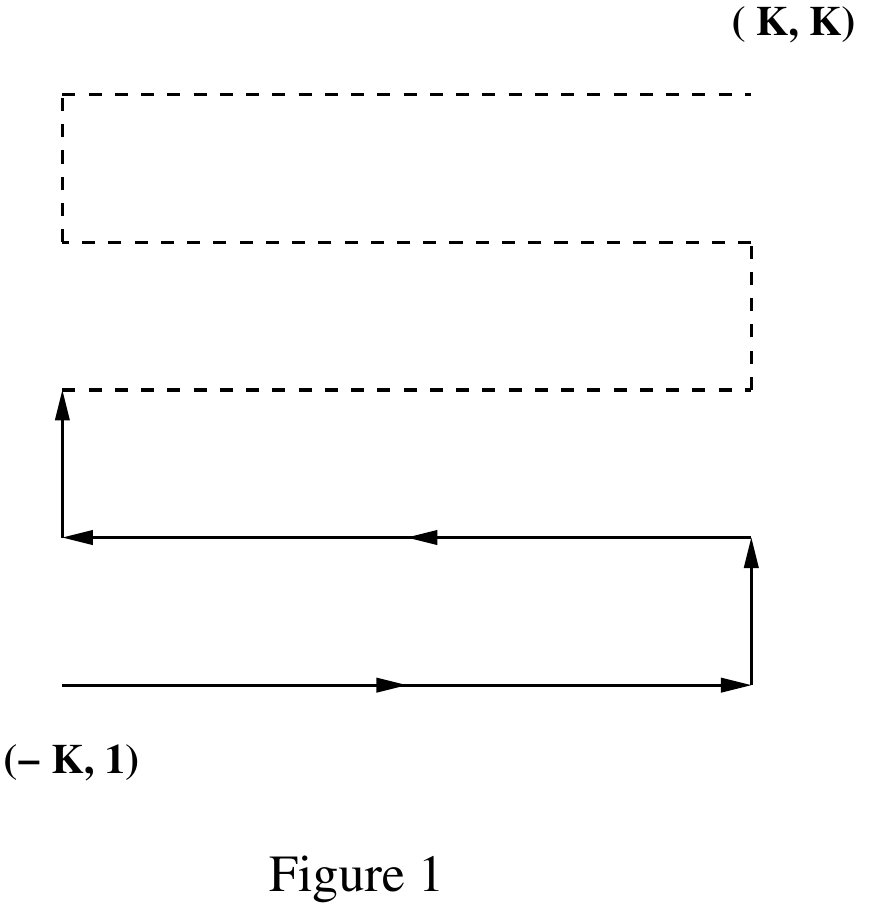} \end{center}
Using summation by parts on $z$, we get
\[\sum_{w,z} E_{u+w, \tilde{u}+z} (\phi_n) 
   \Big{[} \omega(u,w)\omega(\tilde{u},z)-
     \mathbb{E}\big{\lbrace}\omega(u,w)\omega(\tilde{u},z)\big{\rbrace}\Big{]}
\hspace{3cm}\]
\[ = \sum_{w} \bigg{[}\bigg{(} \omega(u,w)\sum_{z}\omega(\tilde{u},z) - \mathbb{E} \Big{\lbrace} \omega(u,w)\sum_{z}\omega(\tilde{u},z)\Big{\rbrace} \bigg{)} \cdot E_{u+w,\tilde{u}+(K,K)}(\phi_n) \bigg{]} \]
\[ - \sum_{w} \bigg{[} \sum_{  \begin{tiny} \begin{array}{c}1\le  z\cdot e_2  \leq K \\
 -K \le  z\cdot e_1  \leq K \\ z \ne (K,K) \end{array}\end{tiny}} \Big{\lbrace} \sum_{\beta \leq z} \omega(u,w)\omega(\tilde{u}, \beta ) -\mathbb{E}\Big{(}\sum_{\beta \leq z} \omega(u,w)\omega(\tilde{u}, \beta )\Big{)}  \Big{\rbrace} \]
\[ \hspace{2cm} \qquad \qquad \qquad  \times \Big{\lbrace}E_{u+w,\tilde{u}+z+1}(\phi_n) - E_{u+w,\tilde{u}+z}(\phi_n )
  \Big{\rbrace}  \bigg{]}.
\]
Write $R_{l,1}= R_{l,1,1}-R_{l,1,2}$, after splitting $\sum_{w,z}$ into the two $\sum_w$ terms above. Similarly, we can show 
\[ R_{l,2}= \sum_{u \cdot e_2=l}\,\, \sum_{ l-K+1 \leq \tilde{u} \cdot e_2 \leq l-1} \,\,\sum_{ z \cdot e_2 \geq l- \tilde{u} \cdot e_2 +1 } P_{r_i \sqrt{n}}^{\omega}(X_{\lambda_l }=u)P_{r_j \sqrt{n}}^{\omega}(\tilde{X} \mbox{ hits } \tilde{u})\omega(\tilde{u},z) \]              
\[\qquad \qquad  \sum_{ \begin{tiny} \begin{array}{c}1\le  w\cdot e_2  \leq K \\
 -K \le  w\cdot e_1  \leq K \\ w \ne (K,K) \end{array}\end{tiny} }\Big{\lbrace} \sum_{\nu \leq w}\big{[}\omega(u,\nu) - \mathbb{E} \omega(u,\nu)\big{]}   \Big{\rbrace} \cdot \Big{\lbrace}E_{u+w,\tilde{u}+z}(\phi_n) - E_{u+w+1,\tilde{u}+z}(\phi_n) \Big{\rbrace}.\]
Do the same for $R_{l,2}^{'}$. Proposition \ref{prop1} will be proved if we can show the following.
\end{proof}

\bigskip

\begin{proposition}
\label{prop3}
\[ \frac{1}{n} \sum_{l=0}^n \mathbb{E}(R_{l,1,1}^2) \longrightarrow 0 \;\; \mbox{ and  } \;\;\frac{1}{n} \sum_{l=0}^n \mathbb{E}(R_{l,1,2}^2) \longrightarrow 0 .\]
Also
\[ \frac{1}{n}  \sum_{l=0}^n \mathbb{E}(R_{l,2}^2) \longrightarrow 0 \;\;\mbox{ and }\;\; \frac{1}{n}  \sum_{l=0}^n \mathbb{E}({R_{l,2}^{'}}^2) \longrightarrow 0.\]
\end{proposition}
\begin{proof}Let us show the first statement of the proposition. Since $\sum_z \omega(\tilde{u},z) = 1$, we have
\[R_{l,1,1}^2= \sum_{u,v}\sum_{\tilde{u},\tilde{v}}P_{r_i \sqrt{n}}^{\omega}(X_{\lambda_l}=u)P_{r_j \sqrt{n}}^{\omega}(\tilde{X}_{\tilde{\lambda}_l}=\tilde{u})P_{r_i \sqrt{n}}^{\omega}(X_{\lambda_l}=v)P_{r_j \sqrt{n}} ^{\omega}(\tilde{X}_{\tilde{\lambda}_l}=\tilde{v})\hspace{4cm}\]
\[ \hspace{1cm}\times\sum_{w,r}\Big{[}\Big{(} \omega(u,w) - \mathbb{E}\omega(u,w)\Big{)}\cdot \Big{(} \omega(v,r) - \mathbb{E}\omega(v,r)\Big{)} \] \[ \hspace{9cm} \cdot E_{u+w,\tilde{u}+(K,K)}(\phi_n)\cdot E_{v+r,\tilde{v}+(K,K)}(\phi_n)\Big{]}.\]
Using summation by parts separately for $w$ and $r$, we get the sum after the $\times$ is 
\[\sum_{  \begin{tiny} \begin{array}{c}1\le  w\cdot e_2, r\cdot e_1  \leq K \\
 -K \le  w\cdot e_1 , r\cdot e_1  \leq K  \end{array}\end{tiny}} \sum_{\alpha \le w} \Big{[}\omega(u,\alpha)- \mathbb{E}\omega(u,\alpha)\Big{]}\cdot \Big{[} E_{u+w,\tilde{u}+(K,K)}(\phi_n)-E_{u+w+1,\tilde{u}+(K,K)}(\phi_n)\Big{]}\]
\[ \hspace{2cm} \times  \sum_{\beta \le r} \Big{[}\omega(v,\beta)- \mathbb{E}\omega(v,\beta)\Big{]}\cdot \Big{[} E_{v+r,\tilde{v}+(K,K)}(\phi_n)-E_{v+r+1,\tilde{v}+(K,K)}(\phi_n)\Big{]}.\]
When we take $\mathbb{E}$ expectation in the above expression, we get $0$ unless $u=v$. Also using Lemma \ref{2sup} below, we have that $E_{u+w,\tilde{u}+(K,K)}(\phi_n)-E_{u+w+1,\tilde{u}+(K,K)}(\phi_n)$ and 
$E_{v+r,\tilde{v}+(K,K)}(\phi_n)-E_{v+r+1,\tilde{v}+(K,K)}(\phi_n)$ are bounded. These observations give us that 
\[ \mathbb{E}(R_{l,1,1}^2) \le C P_{r_i \sqrt{n},r_j \sqrt{n}}(X_{\lambda_l}= \tilde{X}_{\tilde{\lambda}_l} \mbox{ and both walks hit level } l). \]
From computations on Green functions in the previous section (eqns. (\ref{eqn7}) and (\ref{green}) ), we have $\frac{1}{n} \sum_{l=0}^n \mathbb{E}(R_{l,1,1}^2) \to 0$. Let us now show \begin{math} \frac{1}{n}\sum_{l=0}^n \mathbb{E} R_{l,2}^2  \to 0 \end{math}. Now
\[\mathbb{E}R_{l,2}^2=\sum_{u\cdot e_2=v\cdot e_2=l} \sum_{\begin{tiny}\begin{array}{c}l-K+1\le \tilde{u}\cdot e_2,\\ \tilde{v}\cdot e_2 \le l-1\end{array}\end{tiny}} \sum_{\begin{tiny}\begin{array}{c}z\cdot e_2,z'\cdot e_2 \ge \\ l-\tilde{u}\cdot e_2+1\end{array}\end{tiny}} \mathbb{E}\Big{\{}P_{r_i\sqrt{n}}^{\omega}(X_{\lambda_l}=u)P_{r_j\sqrt{n}}^{\omega}(\tilde{X} \mbox{ hits }\tilde{u})\omega(\tilde{u},z)\]
\[\hspace{6cm}  P_{r_i\sqrt{n}}^{\omega}(X_{\lambda_l}=v)P_{r_j\sqrt{n}}^{\omega}(\tilde{X} \mbox{ hits }\tilde{v})\omega(\tilde{v},z')  \Big{\}}\]
\[ \hspace{2cm} \times \sum_{  \begin{tiny} \begin{array}{c}1\le  w\cdot e_2  \leq K \\
 -K \le  w\cdot e_1  \leq K \\ w \ne (K,K) \end{array}\end{tiny}} \sum_{  \begin{tiny} \begin{array}{c}1\le  w'\cdot e_2  \leq K \\
 -K \le  w'\cdot e_1  \leq K \\ w' \ne (K,K) \end{array}\end{tiny}} \mathbb{E}\Big{\{} \sum_{\alpha \le w} [\omega(u,\alpha)-\mathbb{E}\omega(u,\alpha)]\sum_{\alpha' \le  w'}[\omega(v,\alpha')-\mathbb{E}\omega(v,\alpha')]\Big{\}} \]
\[ \hspace{3cm}\times \Big{\{} E_{u+w,\tilde{u}+z}(\phi_n)-E_{u+w+1,\tilde{u}+z}(\phi_n)\Big{\}}\cdot\Big{\{} E_{v+w',\tilde{v}+z'}(\phi_n)-E_{v+w'+1,\tilde{v}+z'}(\phi_n)\Big{\}}.\]
By observing that the expression on the third line of the above equation is zero unless $u=v$ and by using Lemma \ref{2sup} below, it is clear that 
\[ \mathbb{E}R_{l,2}^2 \le C P_{r_i \sqrt{n},r_j \sqrt{n}}(X_{\lambda_l}=\tilde{X}_{\tilde{\lambda}_l} \mbox{ and both walks hit level }l ) \]
and it follows that \begin{math} \frac{1}{n}\sum_{l=0}^n \mathbb{E}R_{l,2}^2 \to 0.\end{math} The remaining parts of the proposition can be similarly proved. This completes the proof of Proposition \ref{prop3} and hence Proposition \ref{STEP 2}
\end{proof}

We have thus shown 
\[ \sum_{k=1}^{[ns]} E\Big{[} \Big{\lbrace} \sum_{i=1}^N\theta_i\Big{(}\frac{M_k^{n,i}-M_{k-1}^{n,i } }{n^{\frac{1}{4}} } \Big{)} \Big{\rbrace} \Big{\lbrace} \sum_{i=1}^N \theta_i\Big{(} \frac{M_k^{n,i}-M_{k-1}^{n,i} } {n^{\frac{1}{4}} }\Big{)}^T\Big{\rbrace}  \Big{\vert} \mathcal{F}_{k-1}\Big{]} \rightarrow h(s)\Gamma\]
where $h(s)$ is as in (\ref{h}). From the left hand side of the above expression, we can conclude that $h(s)$ is nondecreasing. For if $h(s)>h(t)$ for some $s<t$, we have 
\[ \Big(h(t)-h(s)\Big)\Gamma = \lim \limits_{n \to \infty}\sum_{k=[ns]}^{[nt]} E\Big{[} \Big{\lbrace} \sum_{i=1}^N\theta_i\Big{(}\frac{M_k^{n,i}-M_{k-1}^{n,i } }{n^{\frac{1}{4}} } \Big{)} \Big{\rbrace} \Big{\lbrace} \sum_{i=1}^N \theta_i\Big{(} \frac{M_k^{n,i}-M_{k-1}^{n,i} } {n^{\frac{1}{4}} }\Big{)}^T\Big{\rbrace}  \Big{\vert} \mathcal{F}_{k-1}\Big{]}.\]
The left hand side is a nonpositive definite matrix whereas the right hand side is nonnegative definite.
We show that $h$ is H\"older continuous.

\bigskip

\begin{lemma}The function
 \[f(t)=\sum_{i=1}^N\sum_{j=1}^N \theta_i \theta_j \sqrt{t} \int_{0}^{\frac{\overline{\sigma}^2}{c_1}} \frac{1}{\sqrt{2\pi v}}\exp\big{(} - \frac{(r_i-r_j)^2}{2tv}\big{)} dv  \]
has bounded derivative on $(0,1]$ and is continuous on $[0,1]$.
\end{lemma}
\begin{proof} Note that \begin{math}f(0)=\lim \limits_{ t \downarrow 0} f(t) = 0 \mbox{ and } f(1)=\sum_{i=1}^N\sum_{j=1}^N \theta_i \theta_j \int_{0}^{\frac{\overline{\sigma}^2}{c_1}} \frac{1}{\sqrt{2 \pi v}}\exp\big{(} -\frac{(r_i-r_j)^2}{2v}\big{)}dv.\end{math} For $0<s<t$, we have 
\[ f(t)-f(s) \leq C(\sum \theta_i)^2 \max_{i,j} \Big{\lbrace} \sqrt{t}\int_0^{\frac{\overline{\sigma}^2}{c_1}} \frac{1}{\sqrt{2 \pi v}} \exp\big{(} - \frac{(r_i-r_j)^2}{2tv}\big{)}dv\]
\[\hspace{6cm} - \sqrt{s}\int_0^{\frac{\overline{\sigma}^2}{c_1}}\frac{1}{\sqrt{2 \pi v}} \exp\big{(} - \frac{(r_i-r_j)^2}{2sv}\big{)}dv   \Big{\rbrace} . \]
Now for $B>0$ and $0<s<t$
\[\Big{\vert}\sqrt{t}\int_0^{\frac{\overline{\sigma}^2}{c_1}} \frac{1}{\sqrt{2 \pi v}} \exp\big{(} - \frac{B}{2tv}\big{)}dv - \sqrt{s}\int_0^{\frac{\overline{\sigma}^2}{c_1}}\frac{1}{\sqrt{2 \pi v}} \exp\big{(} - \frac{B}{2sv}\big{)}dv\Big{\vert}  \]
\[= \Big[\frac{1}{2\sqrt{u}}\int_0^{\frac{\overline{\sigma}^2}{c_1}} \frac{1}{\sqrt{2\pi v}} \exp\big( -\frac{B}{2uv} \big)dv +\frac{B}{2u^{\frac{3}{2}}}\int_0^{\frac{\overline{\sigma}^2}{c_1}} \frac{1}{v^{\frac{3}{2}}\sqrt{2 \pi}} \exp\big( -\frac{B}{2uv} \big)dv \Big] (t-s)\]
for some $s \leq u\leq t$. Since the right hand side of the above equation  is bounded for $0< u \leq 1$, we are done .
\end{proof} 
Theorem \ref{theorem} is now proved except for 

\bigskip

\begin{lemma}
\label{2sup}
\[ \sup_n \sup_{u \in \mathbb{Z}^2} \sup_{ y \in \lbrace e_1,e_2 \rbrace }\Big{\vert}E_{0,u}(\phi_n ) 
- E_{0,u+y}(\phi_n ) \Big{\vert}  < \infty.\]
\end{lemma}
We first prove the following 

\bigskip

\begin{lemma}
\label{2sup1sup}
 \[ \sup_n \sup_{u \in \mathbb{Z}^2} \sup_{ y \in \lbrace e_1,e_2 \rbrace }\Big{\vert}E_{0,u}(\phi_n ) 
- E_{0,u+y}(\phi_n ) \Big{\vert}  \leq C \sup_n \sup_{ y \in \lbrace \pm e_1, \pm e_2 \rbrace } \Big{\vert}E_{0,0}(\phi_n ) 
- E_{0,y}(\phi_n ) \Big{\vert}.\]
\end{lemma}
\begin{proof} 
\[ \vert E_{0,u}(\phi_n)-E_{0,u+e_1}(\phi_n)\vert = \Big\vert \mathbb{E}\Big{(}E_{0,u}^{\omega}(\phi_n)\Big{)} - \mathbb{E}\Big{(}E_{0,u+e_1}^{\omega}(\phi_n)\Big{)}  \Big\vert \hspace{5cm}\]
\[ = \Big{\vert} \sum_{\begin{tiny}\begin{array}{c}0=x_0,x_1, \cdots x_n \\ u=\tilde{x}_0,\tilde{x}_1, \cdots , \tilde{x}_n \end{array}\end{tiny}} \big{\vert} (i,j):x_i=\tilde{x}_j, x_i \cdot e_2 \le n, \tilde{x}_j \cdot e_2 \le n \big{\vert} \cdot 
\mathbb{E} \Big{\lbrace}\prod_{i=0}^{n-1}\omega(x_i,x_{i+1})\prod_{j=0}^{n-1}\omega(\tilde{x}_j, \tilde{x}_{j+1} )   \Big{\rbrace} \]
\[ - \sum_{\begin{tiny}\begin{array}{c}0=x_0,x_1, \cdots x_n \\ u=\tilde{x}_0,\tilde{x}_1, \cdots , \tilde{x}_n \end{array}\end{tiny}} \big{\vert} (i,j):x_i=\tilde{x}_j+e_1, x_i \cdot e_2,\tilde{x}_j \cdot e_2 \le n \big{\vert}\cdot \mathbb{E} \Big{\lbrace}\prod_{i=0}^{n-1}\omega(x_i,x_{i+1})\prod_{j=0}^{n-1}\omega(\tilde{x}_j+e_1, \tilde{x}_{j+1}+e_1 )   \Big{\rbrace} \Big{\vert}\]
We now split the common sum into two parts $\sum_1$ and $\sum_2$. The sum $\sum_1$ is over all $0=x_0,x_1,\cdots,x_n$ and $u=\tilde{x}_0,\tilde{x}_1,\cdots,\tilde{x}_n$ such that the first level where $x_i=\tilde{x}_j$ occurs is before the first level where $x_i=\tilde{x}_j+e_1$ occurs. Similarly $\sum_2$ is over all $0=x_0,x_1,\cdots,x_n$ and $u=\tilde{x}_0,\tilde{x}_1,\cdots,\tilde{x}_n$ such that $x_i=\tilde{x}_j$ occurs after $x_i=\tilde{x}_j+e_1$. The above expression now becomes 
 \[ \hspace{1cm}\leq \big{\vert} \sum_1 \cdots \big{\vert} + \big{\vert} \sum_2 \cdots \big{\vert} \leq \sup_n \big{\vert}E_{0,0}(\phi_n)-E_{0,e_1}(\phi_n) \big{\vert} + \sup_n \big{\vert}E_{0,0}(\phi_n)-E_{0,-e_1}(\phi_n) \big{\vert}.\hspace{8cm}\] 


\end{proof}
 The proof of Lemma \ref{2sup} will be complete if we show that the right hand side of the inequality in the Lemma \ref{2sup1sup} is finite. We will work with $y=e_1$. Consider the Markov chain on  $\Big{(}\mathbb{Z} \times \{ -(K-1),-(K-2), \dots (K-1) \} \Big) \times \{0,1,\dots K-1\} $ given by 
 \[ Z_k = \Big{(}\tilde{X}_{\tilde{\lambda}_k}-X_{\lambda_k}, \min(X_{\lambda_k} \cdot e_2, \tilde{X}_{\tilde{\lambda}_k}\cdot e_2) -k \Big{)}. \]
This is an irreducible Markov chain (it follows from assumption \ref{model} (iii)). For $z$ in the state space of the $Z$-chain, define the first hitting time of $z$ by $T_z := \inf\{k \ge 1: Z_k =z \}$. For $z,w$ in the state space, define \begin{math} G_n(z,w)=\sum_{k=0}^n P(Z_k=w \vert Z_0=z) \end{math}. Now note that
\[ G_n\Big{[} \big{(}(0,0),0\big{)},\big{(}(0,0),0\big{)}\Big{ ]} \leq E_{\big{(}(0,0),0\big{)}} \Big{[}\sum_{k=0}^{T_{((1,0),0)}} \mathbb{I}_{\lbrace Z_k =  \big{(}(0,0),0\big{)}\rbrace} \Big{]} + G_n\Big{[}\big{(}(1,0),0 \big{)},\big{(}(0,0),0\big{)} \Big{]}.\]
\[ G_n\Big{[} \big{(}(1,0),0\big{)},\big{(}(0,0),0\big{)}\Big{ ]} \leq E_{\big{(}(1,0),0\big{)}} \Big{[}\sum_{k=0}^{T_{((0,0),0)}}\mathbb{I}_{\lbrace Z_k =  \big{(}(0,0),0\big{)}\rbrace} \Big{]} + G_n\Big{[}\big{(}(0,0),0\big{)},\big{(}(0,0),0\big{)} \Big{]} .\]
Since both expectations are finite by the irreducibility of the Markov chain $Z_k$,
\[ \sup_n \Big{\vert}G_n\Big{(} \big{(}(0,0),0\big{)},\big{(}(0,0),0\big{)} \Big{)}-G_n\Big{(} \big{(}(1,0),0\big{)},\big{(}(0,0),0\big{)} \Big{)}\Big{\vert} < \infty. \]
The proof of Lemma \ref{2sup} is complete by noting that 
\[E_{0,0}(\phi_n)= G_n\Big{(} \big{(}(0,0),0\big{)},\big{(}(0,0),0\big{)} \Big{)} \mbox{ and }E_{0,e_1}(\phi_n)= G_n\Big{(} \big{(}(1,0),0\big{)},\big{(}(0,0),0\big{)} \Big{)}.\qed \]

\bigskip

\section{Appendix}
\begin{proof} \textbf{of Lemma \ref{MCLT}} We just modify the proof of the Martingale Central Limit Theorem in Durrett~\cite{durr}(Theorem 7.4) and Theorem 3 in \cite{rsptrf}. First let us assume that we are working with scalars, that is $ d=1$ and $\Gamma = \sigma^2$. Suppose that $h$ is H\"older continuous with parameter $\alpha$, that is $\vert h(x)-h(y) \vert  \leq D\vert x- y \vert ^{\alpha}$. We first modify Lemma 6.8 in \cite{durr}.
\end{proof}

\bigskip

\begin{lemma}
\label{lemmaMCLT}
 Let $\tau_{m}^{n}$, $1 \leq m \leq n$, be a triangular array of increasing random variables, that is $\tau_{m}^n \leq \tau_{m+1}^n$. Also assume that $\tau_{[ns]}^{n} \to h(s)$  in probability for each $s \in [0,1]$. Let 
\[S_{n,(u)}= \left \{ \begin{array}{ll}B(\tau_{m}^{n}) & \mbox{for $u=m \in \{0,1,\cdots,n\}$};\\
                 \mbox{linear for $u \in [m-1,m]$}& \mbox{ when $m \in \{ 1,2,\cdots,n\}$}. \end{array} \right.\]
We then have     
\[ \vert \vert S_{n,(n\cdot)} - B\big{(}h(\cdot)\big{)} \vert \vert \to 0  \mbox{ in probability } \]
where $\vert \vert \cdot \vert \vert$ is the sup-norm of $C_{[0,1]}$, the space of  continuous functions in $[0,1]$. 
\end{lemma}
\begin{proof} Since $B$ is uniformly continuous on $[0,1]$, given $\epsilon >0$, we can find a $\delta >0$ such that $\frac{1}{\delta}$ is an integer and \\
\begin{enumerate}
\item
\[ P( \vert B_t - B_s \vert < \epsilon \mbox{ for all } 0 \leq s,t\leq 1 \mbox{ such that } \vert t-s  \vert < 2D \delta^{\alpha} ) > 1- \epsilon \]
and for $n \geq N_{\delta} $,
\item
\[ P\big{(} \vert \tau_{[n k \delta]}^{n} - h(k\delta) \vert < D\delta^{\alpha} \mbox{ for } k=1,2, \cdots \frac{1}{\delta}\big{)} \geq 1 - \epsilon .\]
\end{enumerate}
For $s \in \big{(} (k-1)\delta , k \delta \big{)}$
\begin{eqnarray*}
 \tau_{[ns]}^{n}-h(s) &\geq& \tau_{[n(k-1)\delta]}^n - h(k \delta)\\
&=& \tau_{[n(k-1)\delta]}^n - h\big{(} (k-1)\delta \big{)} + \big{[}h\big{(} (k-1)\delta \big{)} - h\big{(} k\delta \big{)} \big{]}.
\end{eqnarray*}
Also
\begin{eqnarray*}
 \tau_{[ns]}^{n}-h(s) &\leq& \tau_{[nk\delta]}^n - h\big{(}(k-1) \delta\big{)}\\
&=& \tau_{[nk\delta]}^n - h\big{(} k\delta \big{)} + \big{[}h (k\delta ) - h\big{(} (k-1)\delta \big{)} \big{]}.
\end{eqnarray*}
By H\"older continuity, we have
\[h (k\delta ) - h\big{(} (k-1)\delta \big{)} \leq D \delta^{\alpha} .\]
Thus,
\[ \tau_{[n(k-1)\delta]}^n-h\big{(} (k-1)\delta \big{)} - D \delta^{\alpha} \leq \tau_{[ns]}^{n}-h(s) \leq 
   \tau_{[nk\delta]}^n - h\big{(} k\delta \big{)} + D \delta^{\alpha}. \]
From this we can conclude that for $n \geq N_{\delta}$,
\[ P\Big{(} \sup_{0 \leq s \leq 1} \vert \tau_{[ns]}^n - h(s) \vert < 2D\delta^{\alpha} \Big{)} \geq 1 - \epsilon .\]
When the events in $(i)$ and $(ii)$ occur,
\[ \vert S_{n,m} - B\big{(}h(\frac{m}{n})\big{)} \vert = \vert B(\tau_{m}^n) - B\big{(}h(\frac{m}{n}) \big{)} \vert < \epsilon .\]
For $t= \frac{m+\theta}{n}$, $0<\theta < 1$, notice that
\begin{eqnarray*}
 \vert S_{n,(nt)} - B(h(t)) \vert &\leq& (1-\theta)\vert S_{n,m} - B\big{(}h(\frac{m}{n})\big{)} + \theta \vert S_{n,m+1} - B\big{(}h(\frac{m+1}{n})\big{)} \vert \\
&+& (1- \theta)\vert B\big{(}h(\frac{m}{n})\big{)} - B\big{(}h(t)\big{)} \vert + \theta \vert B\big{(}h(\frac{m+1}{n})\big{)} - B\big{(}h(t)\big{)} \vert .
\end{eqnarray*}
The sum of the first two terms is $\leq \epsilon$ in the intersection of the events in $(i)$ and $(ii)$. Also for $n \ge M_{\delta}$, we have $\vert \frac{1}{n} \vert < \delta$ and hence \begin{math} \vert h(\frac{m}{n})-h(t) \vert < 2D \delta^{\alpha}, \vert h(\frac{m+1}{n})-h(t) \vert < 2D \delta^{\alpha}
\end{math}. Hence the sum of the last two terms is also $\le \epsilon$ in the intersection of the events in $(i)$ and $(ii)$.
Choosing $\delta$ appropriately, we get that for $n \geq \max(N_{\delta}, M_{\delta}) $,
\[ P \Big{(} \vert \vert S_{n,(n\cdot)} - B\big{(}h(\cdot)\big{)} \vert \vert < 2 \epsilon \Big{)} \geq 1- 2 \epsilon . \qquad \]
\end{proof}
\vspace{0.5cm}
We will also need the following lemma later whose proof is very similar to that of the above lemma.

\bigskip

\begin{lemma}
For increasing random variables $\tau_{m}^n$, $1 \leq m \leq [n(1-s)]$ such that $\tau_{[nt]}^n \to h(t+s) - h(s)$ in probability for each $t \in [0,1-s]$, we have
\[ \big\vert \big\vert \, S_{n,(n \cdot)} - \big{[}B\big{(}h(s+ \cdot)\big{)}- B\big{(}h(s)\big{)} \big{]} \,\big \vert \big \vert  \to 0 \hspace{1cm} \mbox{ in probability}. \]

\end{lemma}
\vspace{0.5cm}
The statements of Theorems 7.2 and Theorems 7.3 in \cite{durr} are modified for our case by replacing $V_{n,[nt]} \to t\sigma^2$ by $V_{n,[nt]} \to h(t)\sigma^2$. The proofs are almost the same as the proofs in \cite{durr}. We have thus proved the theorem for the case when $d=1$. To prove the vector valued version of the theorem, we refer to Theorem 3 of \cite{rsptrf}. Lemma 3 in \cite{rsptrf} now becomes 
\[ \lim \limits_{ n \to \infty} \frac{E(f(\theta \cdot S_n(s+ \cdot)- \theta \cdot S_n(s) ) Z_n ) }{E(Z_n)} = E\big{[}f\big{(}C_{\theta}(s+ \cdot)\big{)}\big{]} \]
where $C_{\theta}(t)= B_{\theta}(h(t))$ and $B_{\theta}$ is a 1 dimensional B.M. with variance $\theta^{T} \Gamma \theta $. The only other fact we will need while following the proof of Lemma 3 in \cite{rsptrf} is 

\bigskip

\begin{lemma}
 \label{ctsversion}
A one dimensional process $X$  with the same finite dimensional distribution as $C(t)=B(h(t))$ has a continuous version.
\end{lemma}
\begin{proof} We check Kolmogorov's continuity criterion. Let $\beta =[\frac{1}{\alpha}]+1$.
\begin{eqnarray*}
 E \Big{(}\vert X_t - X_s \vert^{\beta} \Big{)} &=& E \Big{(}\vert B_{h(t)} - B_{h(s)} \vert^{\beta} \Big{)} \\
&\leq& \vert h(t)- h(s) \vert^{\beta} E(Z^{\beta})\\
&\leq& C\vert t-s \vert ^{\alpha\big{(}[\frac{1}{\alpha}]+1\big{)}} 
\end{eqnarray*}
where $Z$ is a standard Normal random variable. 
\end{proof}
\vspace{0.4cm}
\textbf{Acknowledgement}\\
This work is part of my Ph.D. thesis. I would like to thank my advisor Timo Sepp\"al\"ainen for suggesting this problem and for many valuable discussions.

\end{document}